\documentclass[12pt]{article}

\usepackage[T1]{fontenc}
\usepackage[textwidth=16.6cm,textheight=23.2cm,centering]{geometry}
\usepackage{amsmath,amssymb,amsthm,mathtools}
\usepackage{newtxtext,newtxmath}
\usepackage{microtype}
\usepackage[hidelinks]{hyperref}
\usepackage[nameinlink,capitalise,noabbrev]{cleveref}

\newtheorem{theorem}{Theorem}[section]
\newtheorem{lemma}[theorem]{Lemma}
\newtheorem{fact}[theorem]{Fact}
\newtheorem{proposition}[theorem]{Proposition}
\newtheorem{definition}[theorem]{Definition}

\theoremstyle{remark}

\crefname{theorem}{Theorem}{Theorems}
\crefname{lemma}{Lemma}{Lemmas}
\crefname{proposition}{Proposition}{Propositions}
\crefname{corollary}{Corollary}{Corollaries}

\newcommand{\bbP}{\mathbb{P}}
\newcommand{\bbE}{\mathbb{E}}

\newcommand{\ex}{\operatorname{ex}}

\renewcommand{\labelenumi}{\theenumi}

\newcommand{\hide}[1]{}
\newcommand{\e}{\varepsilon}

\begin{document}

\title{An improved bound on the minimum size of\\ Tur\'an $(r+1,r)$-systems}
\author{Jun Gao\footnote{Academy of Mathematics and Systems Science, Chinese Academy of Sciences, Beijing 100084, China. Email:
jungao@amss.ac.cn}
\and
Peiru Kuang\footnote{School of Mathematical Sciences, Shanghai Jiao Tong University, Shanghai 200240, China. Email: peiru\_k@sjtu.edu.cn} 
\and
Oleg Pikhurko\footnote{Mathematics Institute and DIMAP,
University of Warwick,
Coventry CV4 7AL, UK. Email: o.pikhurko@warwick.ac.uk}
\and
Yan Wang\footnote{School of Mathematical Sciences, Shanghai Jiao Tong University, Shanghai 200240, China. Email: yan.w@sjtu.edu.cn}}

\date{}

\maketitle

\begin{abstract}
For positive integers \(n\ge s>r\), let \(T(n,s,r)\) denote the minimum number of edges in an \(r\)-uniform hypergraph on \(n\) vertices such that every \(s\)-set of vertices contains at least one edge. A simple averaging argument shows that the ratio $T(n,s,r)/\binom nr$ is non-decreasing in $n$ and we denote its limit as $n\to\infty$ by $t(s,r)$.

The case $s=r+1$ has a rich history, with the previously best known asymptotic bounds for $r\to\infty$  being $1\le r\cdot t(r+1,r)\le 4.91...$\,.
In this paper, we present a simple probabilistic construction which shows that $(r+2)\cdot t(r+1,r)\le 4$ for every $r\ge1$. We also derandomise it and discuss applications to covering codes.
\end{abstract}

\section{Introduction}

Given an integer \(r\ge 1\), an \emph{\(r\)-uniform hypergraph} (or briefly an \emph{\(r\)-graph}) is a family of \(r\)-subsets of a vertex set. For integers
\(n\ge s>r\ge 1\), a \emph{Tur\'an \((n,s,r)\)-system} is an \(r\)-graph
\(G\subseteq\binom{[n]}r\) such that every \(s\)-subset of \([n]\)
contains at least one edge of \(G\). The minimum possible size of such an $r$-graph is denoted by \(T(n,s,r)\). Observe that 
$$
 T(n,s,r)=\binom nr-\ex(n,K_s^{(r)}),
$$
where \(\ex(n,K_s^{(r)})\) is the maximum number of edges in an \(n\)-vertex \(r\)-graph without \(K_s^{(r)}\), the complete $r$-graph on $s$ vertices.
A simple averaging argument shows that
\begin{equation}\label{eq:monotone}
\frac{T(n,s,r)}{\binom nr}\le \frac{T(n+1,s,r)}{\binom {n+1}r},\quad\text{for every $n\ge s$.}
\end{equation}
In particular, it follows that the limit
$$
 t(s,r):=\lim_{n\to\infty}\frac{T(n,s,r)}{\binom nr}
$$
exists and, by~\eqref{eq:monotone}, we have $T(n,s,r)\le t(s,r) \binom nr$ for every $n\ge s$. 

In the trivial case when $r = 1$, it holds that $T(n,s,1) = n-s+1$ for any $n \geq s> 1$. For \(r=2\), the problem is resolved by Tur\'an's theorem \cite{Turan1941}, with the case when \(s=3\) proved earlier by Mantel \cite{Mantel1907}. For \(s>r\ge3\), however, the value of \(t(s,r)\) is not known for any pair \((s,r)\). Erd\H{o}s
\cite[Section~III.1]{Erdos1981} offered \$500 for determining $t(s,r)$ for a single pair $s>r\ge3$ and \$1000 for resolving the problem completely, but these remain unclaimed.
Tur\'an and other researchers conjectured (see e.g.~\cite[Conjecture~5]{Sidorenko1995}) that \(t(s,3)=4/(s-1)^2\) for every \(s\ge4\). In the first open case, the conjecture asserts that \(t(4,3)=4/9\), while the current best lower bound is \(t(4,3)\ge0.438...\) due to Razborov \cite{Razborov2010}. Much less is known for uniformity at least four. For the first such case, Giraud~\cite{Giraud1990} constructed a family showing that
$
t(5,4)\le {5}/{16}=0.3125,
$
whereas Markstr\"om~\cite{Markstrom2009} proved that
$T(17,5,4)\ge627$. This directly implies that $t(5,4)\ge 627/\binom{17}{4}=0.2634...$ while an integer rounding trick for $n=18$ slightly improves this to $t(5,4)\ge 807/\binom{18}4>0.2637$, see~Sidorenko~\cite[Section~2]{Sidorenko2021}.
We refer to
\cite{deCaen1994,Furedi1991,Keevash2011,Sidorenko1995}
for further background on hypergraph Tur\'an problems.

In this paper, we focus on the case when \(s=r+1\). A simple double-counting argument (or~\eqref{eq:monotone}) gives \(t(r+1,r)\ge1/(r+1)\), while the stronger estimate \(t(r+1,r)\ge1/r\) was independently obtained by de Caen \cite{deCaenExtension1983}, Sidorenko \cite{Sidorenko1982}, and Tazawa and Shirakura \cite{TazawaShirakura1983}. Chung and Lu \cite{ChungLu1999} sharpened this bound for odd $r$, and Lu and Zhao \cite{LuZhao2009} for even $r$.

In the other direction, 
the successive upper bounds \(O(r^{-1/2})\), \((1+2\ln r)/r\), \((\ln r+O(1))/r\), and \((1+o(1))\ln r/(2r)\) on $t(r+1,r)$  as $r\to\infty$ were shown respectively by Sidorenko \cite{Sidorenko1981}, Kim and Roush \cite{KimRoush1983}, Frankl and R\"odl \cite{FranklRodl1985}, and Sidorenko \cite{Sidorenko1997}. De Caen \cite[p.~190]{deCaen1994} conjectured that \(r\cdot t(r+1,r)\to\infty\) and
offered 500 Canadian dollars for its resolution. Pikhurko~\cite{Pikhurko2025} disproved the conjecture by showing that $(r+1)\cdot t(r+1,r)\le 6.239$
for every \(r\ge1\) and that \(6.239\) can be replaced by \(4.911\) for all sufficiently large \(r\).

The construction in \cite{Pikhurko2025} is recursive. In this paper, we give a simple probabilistic construction yielding the following improvement.

\begin{theorem}\label{thm:main}
Let $n>r\ge1$ be integers. Then
$$
 T(n,r+1,r)
 \le
 \frac{2}{1+\lfloor(r+1)/2\rfloor}\binom nr
 \le
 \frac{4}{r+2}\binom nr.
$$
\end{theorem}

We also give an explicit deterministic construction that gives the same upper bound on $t(r+1,r)$.

\begin{theorem}\label{thm:main-explicit}
Let $r\ge2$ be fixed. For all sufficiently large $n$, there is an explicit construction of a Tur\'an $(n,r+1,r)$-system $G$ such that
$$
 |G|
 \le
 \left(\frac{2}{1+\lfloor (r+1)/2\rfloor}+o_n(1)\right)\binom nr.
$$
Moreover, it can be decided in time polynomial in $r\log_2 n$ whether a given $r$-subset of $[n]$ is an edge of $G$.
\end{theorem}

The remainder of the paper is organised as follows. Section~\ref{se:codes} discusses applications of the new bounds to covering codes.
We prove Theorems \ref{thm:main} and \ref{thm:main-explicit} in Sections~\ref{sec:proof} and~\ref{sec:explicit} respectively. 
Section~\ref{sed:LargerGap} extends the construction of Theorem~\ref{thm:main} to Tur\'an $(s,r)$-systems with $s\ge r+2$.

\hide{
\medskip
\noindent
\textbf{Organisation.}
The remainder of the paper is organised as follows. In
\cref{sec:proof}, we prove \cref{thm:main}. In
\cref{sec:explicit}, we give an explicit ordering and prove \cref{thm:main-explicit}. Finally, in \cref{sed:LargerGap}, we discuss an extension to larger gaps.
}

\section{Applications to covering codes}
\label{se:codes}

Tur\'an systems have direct connections to various kinds of covering codes. For example, Pikhurko, Verbitsky and Zhukovskii in \cite[Lemma~5]{PikhurkoVerbitskyZhukovskii2025} and a subsequent revision of their preprint~\cite[Theorem~18]{PikhurkoVerbitskyZhukovskii2026arXiv} showed that $t(s,r)$ is equal to the minimum asymptotic density as $n\to\infty$ of an \emph{$(s-r)$-insertion covering code} on $[n]^r$, which is a collection of $r$-words over the alphabet $[n]$ such that every $s$-word can be obtained from some codeword by inserting $s-r$ symbols. Thus our improved bounds on $t(r+1,r)$ directly yield improved bounds on code densities; we refer to \cite{LenzRashtchianSiegelYaakobi2021,PikhurkoVerbitskyZhukovskii2025} for overviews of the previously known bounds.
A similar conclusion holds for the so-called \emph{stopping redundancy}, introduced by Schwartz and Vardy~\cite{SchwartzVardy2006}. For a \emph{maximum distance separable (MDS) linear $[n,n-d+1,d]$-code}, Han and Siegel~\cite[Theorem~22]{HanSiegel2007} showed that, for every fixed $d$, its stopping redundancy is $(1+O(1/n))T(n,d-1,d-2)$.

As our final example, let us discuss asymmetric binary covering codes. Their formal connection to Tur\'an systems was not explicitly stated in the literature, so we present it here. 
An \emph{asymmetric binary covering code} of radius $\rho$ is a family $\mathcal D\subseteq2^{[n]}$ such that, for every $A\subseteq[n]$, there is $D\in\mathcal D$ satisfying $A\subseteq D$ and $|D\setminus A|\le \rho$, that is, every subset of $[n]$ can be obtained from a codeword by deleting at most $\rho$ elements. Let $K^+(n,\rho)$ denote the minimum cardinality of such a code. The following result shows how upper bounds on $t(r+\rho,r)$ translate into upper bounds on $K^+(n,\rho)$.

\begin{lemma}\label{lm:AsymmetricCodes} 
For every fixed integer $\rho\ge 1$ it holds that
\begin{equation}
 K^+(n,\rho)
 \le
 \left(\rho!\, 2^\rho \mu+o(1)\right)\frac{2^n}{n^\rho},\quad \text{as $n\to\infty$},\label{eq:8}
\end{equation}
where $\mu:=\limsup_{r\to\infty}\binom{r+\rho}{\rho}\cdot t(r+\rho,r)$.
 \end{lemma}

\begin{proof}
Given $\e>0$, fix an integer $r_0$ such that $\binom{r+\rho}{\rho}\cdot t(r+\rho,r)\le \mu+\e$ for every $r> r_0$.

Let $n$ be sufficiently large and let $\mathcal G_r\subseteq\binom{[n]}r$, for each $1\le r\le n-\rho$, be a Tur\'an $(n,r+\rho,r)$-system of size $T(n,r+\rho,r)$, which is at most $t(r+\rho,r)\binom {n}{r}$ by~\eqref{eq:monotone}. Set $\mathcal G_0:=\{\emptyset\}$. Define
$$
 \mathcal D
 :=
 \bigl\{[n]\setminus E:
 E\in\mathcal G_r\text{ for some }0\le r\le n-\rho\bigr\}.
$$
Take any $A\subseteq[n]$, and let $B:=[n]\setminus A$. If $|B|\le \rho$, then $A$ can be obtained from $[n]\in\mathcal D$ by removing at most $\rho$ elements. Otherwise, there is an edge $E\in\mathcal G_{|B|-\rho}$ such that $E\subseteq B$. Hence $[n]\setminus E$ contains $A$ and differs from it in at most $\rho$ elements. Thus $\mathcal D$ is an asymmetric binary covering code of
radius~$\rho$.

By the definitions of $r_0$ and $\mu$, we have
$$
 |\mathcal D|
 \le \sum_{r=0}^{r_0}\binom nr+
  \sum_{r=0}^{n}\frac{\mu+\e}{\binom{r+\rho}{\rho}}\, \binom nr
 \le O(n^{r_0})+(\mu+\e)\frac{ 2^{n+\rho}}{\binom{n+\rho}{\rho}}.
$$
Since $\e>0$ was arbitrary, this implies the bound claimed by the lemma.\end{proof}

Cooper, Ellis, and Kahng \cite{CooperEllisKahng2002} proved in particular that, for every $n\ge1$,
$$
 \left\lceil
 \frac{n\cdot 2^{n+1}+2}{(n+1)(n+2)}
 \right\rceil
 \le K^+(n,1)
 \le
 \frac{\gamma_1\cdot 2^{n}}{n},
$$
for some constant $\gamma_1$. Their proof of~\cite[Corollary 8]{CooperEllisKahng2002}, as stated, requires \(\alpha_1\ge 2\) and \(\beta_1\ge 10.713...\) (the unique root of $\beta=4(\ln (\beta/2)+1)$ with $\beta\ge 4$) while the resulting constant is \(\gamma_1=2\beta_1\). On the other hand, \cref{thm:main} and Lemma~\ref{lm:AsymmetricCodes} imply that
\[ K^+(n,1)
 \le
 \left(8+o(1)\right)\frac{2^n}{n},\quad \text{as $n\to\infty$}.\]

\section{Proof of Theorem~\ref{thm:main}}\label{sec:proof}

In this section, we prove \cref{thm:main}. Let $n>r\ge1$, and set $k=\lfloor(r+1)/2\rfloor$. Our aim is to prove that
$$
 T(n,r+1,r)
 \le
 \frac{2}{k+1}\binom nr.
$$
If $k=1$, the desired inequality is just the trivial bound
$T(n,r+1,r)\le\binom nr$. Thus we may assume that $k \ge 2$.

First, we need some notation.
Let $\le$ denote the standard order on $[n]$. For a subset $X=\{x_1<\cdots<x_m\}$ of $[n]$ and an integer $0\le \ell\le m$, define
\begin{equation}\label{eq:DefLR}
 L_\ell(X):=\{x_1,\ldots,x_\ell\}
 \quad\text{and}\quad
 R_\ell(X):=\{x_{m-\ell+1},\ldots,x_m\}
\end{equation}
 to be its initial and final $\ell$-subsets.

\begin{definition}[Set system $G(n,r,\preceq)$]\label{def:set-system}
For a total order $\preceq$ on $\binom{[n]}{k-1}$, let the $r$-graph $G(n,r,\preceq)$ consist of those 
$e\in\binom{[n]}r$ for which at least one of the following conditions is satisfied:
\begin{enumerate}
\renewcommand{\labelenumi}{(\roman{enumi})}
\item\label{cond:left}
$L_{k-1}(e) \preceq Y$ for every $Y \in \binom{R_{k}(e)}{k-1}$, or
\item\label{cond:right} 
$R_{k-1}(e) \preceq Z$ for every $Z \in \binom{L_k(e)}{k-1}$.
\end{enumerate}
\end{definition}

Condition (i) says that the initial $(k-1)$-subset of $e$ precedes all $(k-1)$-subsets of its final $k$-set under $\preceq$. Condition (ii) is the analogous requirement with the initial and final parts interchanged.

We first verify that the construction always gives a desired Tur\'an system, regardless of the choice of the total order $\preceq$.
\begin{lemma}\label{lem:turan-system}
For any total order $\preceq$ on $\binom{[n]}{k-1}$,
the \(r\)-graph $G(n,r,\preceq)$ is a Tur\'an
\((n,r+1,r)\)-system.
\end{lemma}
\begin{proof}
It is enough to show that every $(r+1)$-subset $X=\{x_1<\cdots<x_{r+1}\}$ of $[n]$ contains a member of $G(n,r,\preceq)$. 
Since $2k\le r+1$, the sets $L_k(X)$ and $R_k(X)$ are disjoint. Let $W$ be the minimum member of
$$
 \binom{L_k(X)}{k-1}\cup\binom{R_k(X)}{k-1}
$$
with respect to $\preceq$.

Assume by symmetry that $W\in\binom{L_k(X)}{k-1}$. If $r+1=2k$, set $e:=W\cup R_k(X)$. If $r+1=2k+1$, set
$e:=W\cup\{x_{k+1}\}\cup R_k(X)$. In both cases $|e|=r$,
$L_{k-1}(e)=W$, and $R_{k}(e)=R_k(X)$. Moreover, the minimality of $W$ gives $W\preceq Y$ for every $Y\in\binom{R_k(e)}{k-1}$.
Thus $e$ satisfies Condition~\ref{cond:left} of Definition \ref{def:set-system} and thus belongs to $G(n,r,\preceq)$.\hide{
If $W\in\binom{R_k(X)}{k-1}$, we argue symmetrically. Namely, take $e=L_k(X)\cup W$ when $r+1=2k$, and take
$e=L_k(X)\cup\{x_{k+1}\}\cup W$ when $r+1=2k+1$. Then
$L_{k}(e)=L_k(X)$, $R_{k-1}(e)=W$, and $e$ satisfies Condition~\ref{cond:right}. In either case, $X$ contains an edge of $G(n,r,\preceq)$.}
\end{proof}

We next show that a uniformly random choice gives the desired bound in expectation.

\begin{lemma}\label{lem:random}
Suppose that $k>1$, and let $\preceq$ be a uniformly random total order on $\binom{[n]}{k-1}$. Then 
\[
 \bbE\,|G(n,r,\preceq)|
 \le \frac{2}{k+1}\binom nr.
\]
\end{lemma}
\begin{proof}
Fix $e\in\binom{[n]}r$. The set $L_{k-1}(e)$ is disjoint from $R_{k}(e)$, so the $(k+1)$ members of
\[
 \{L_{k-1}(e)\}\cup\binom{R_{k}(e)}{k-1}
\]
are distinct. Each of them is equally likely to be the minimum one under a uniformly random total order $\preceq$. Hence the probability that Condition~\ref{cond:left} holds is exactly $\frac{1}{k+1}$. By symmetry the same probability applies to Condition~\ref{cond:right}. Thus the probability that $e\in G(n,r,\preceq)$ is at most $2/(k+1)$. (In fact, it is equality since these two events are disjoint: Condition~\ref{cond:left} implies $L_{k-1}(e)\prec R_{k-1}(e)$, whereas Condition~\ref{cond:right} implies $R_{k-1}(e)\prec L_{k-1}(e)$.)
Therefore, the lemma follows by summing these probabilities over all
$e\in\binom{[n]}r$.
\end{proof}
Now we are ready to prove Theorem~\ref{thm:main}.
\begin{proof}[Proof of Theorem~\ref{thm:main}]
Choose a total order $\preceq$ on $\binom{[n]}{k-1}$ uniformly at random. By Lemma~\ref{lem:random}, there is a total order $\preceq'$ such that
\[
 |G(n,r,\preceq')|
 \le\frac{2}{k+1}\binom nr.
\]
By Lemma~\ref{lem:turan-system}, the $r$-graph
$G(n,r,\preceq')$ is a Tur\'an $(n,r+1,r)$-system.\end{proof}

\section{Proof of Theorem \ref{thm:main-explicit}}\label{sec:explicit}

We now give an explicit order which attains the same asymptotic upper bound as the random order. Fix $r\ge3$ and, as before, set $k:=\lfloor(r+1)/2\rfloor$. For all sufficiently large $n$, choose an integer $q$ such that
$\ln n\le q\le2\ln n$ and $\gcd(q,k-1)=1$. Such a choice is possible for all sufficiently large $n$: For example, we can take the unique integer congruent to $1$ modulo $k-1$ in the interval $[\lceil\ln n\rceil,\lceil\ln n\rceil+k-2]$.

We represent the residue classes modulo $q$ by
$[q]=\{1,\ldots,q\}$, with the zero residue represented by $q$. Define $f\colon\binom{[n]}{k-1}\to[q]$ by letting $f(S)$ be the unique element of $[q]$ such that
\[
 f(S)\equiv\sum_{x\in S}x\pmod q.
\]
Let $\preceq_{\rm lex}$ be the ordinary lexicographic order on the increasing $(k-1)$-tuples representing the members of $\binom{[n]}{k-1}$. For $S,T\in\binom{[n]}{k-1}$, define $S\preceq_q T$ if either $f(S)<f(T)$, or $f(S)=f(T)$ and $S\preceq_{\rm lex}T$. Thus $\preceq_q$ is a total order; the lexicographic order only breaks ties between sets with the same $f$-value.

We need two elementary facts about residues. The first one records the asymptotic equidistribution of the residues of the order statistics of a uniformly chosen set.

\begin{fact}\label{fact:residues}
Let $\ell\ge1$ be fixed, let $I$ be an interval of $N\to\infty$ consecutive integers, and let $q=o(N)$. If
$\{w_1<\cdots<w_\ell\}$ is chosen uniformly at random from
$\binom I\ell$, then, uniformly over
$(\xi_1,\ldots,\xi_\ell)\in[q]^\ell$,
$$
 \bbP\bigl(w_i\equiv\xi_i\pmod q\text{ for every }i\in[\ell]\bigr)
 =q^{-\ell}\left(1+O_\ell\left(\frac qN\right)\right).
$$
\end{fact}

\begin{proof}
For fixed residues $\xi_1,\ldots,\xi_\ell$, a direct induction on $\ell$ gives
$$
 \bigl|\{(w_1,\ldots,w_\ell)\in I^\ell:
 w_1<\cdots<w_\ell,\ w_i\equiv\xi_i\pmod q\}\bigr|
 =\frac{N^\ell}{\ell!q^\ell}
  +O_\ell\left(\frac{N^{\ell-1}}{q^{\ell-1}}+1\right).
$$
Indeed, for $\ell=1$ this is the fact that each residue class occurs $N/q+O(1)$ times in $I$. For the induction step, sum the estimate for $w_1<\cdots<w_{\ell-1}<w_\ell$ over the possible $N/q+O(1)$ values of $w_\ell$ in its prescribed residue class. The standard estimate for a sum of a fixed power over an arithmetic progression gives the
displayed main term and error term. Since $\binom N\ell=N^\ell/\ell!+O_\ell(N^{\ell-1})$, division by $\binom N\ell$ proves the fact.
\end{proof}

\begin{lemma}\label{lem:deletion-map}
Suppose that $\gcd(q,k-1)=1$. The map
\[
 \begin{split}
 \Phi\colon(\mathbb Z/q\mathbb Z)^k&\longrightarrow
 (\mathbb Z/q\mathbb Z)^k,\\
 (w_1,\ldots,w_k)&\longmapsto
 \left(\sum_{j\ne1}w_j,\ldots,\sum_{j\ne k}w_j\right)
 \end{split}
\]
is a bijection.
\end{lemma}

\begin{proof}
Let $(y_1,\ldots,y_k)=\Phi(w_1,\ldots,w_k)$ and set
$s:=\sum_{j=1}^k w_j$. Then $y_i=s-w_i$ for every $i\in [k]$, and hence
$$
 \sum_{i=1}^k y_i=(k-1)s.
$$
Since $k-1$ is invertible modulo $q$, the value of $s$ is determined by the vector $(y_1,\ldots,y_k)$, and then
$$
 w_i=(k-1)^{-1}\sum_{j=1}^k y_j-y_i,
 \quad\text{for every }i\in[k].
$$
This gives the inverse of $\Phi$.
\end{proof}

\begin{lemma}\label{lem:explicit}
The $r$-graph $G(n,r,\preceq_q)$ is a Tur\'an $(n,r+1,r)$-system and satisfies
$$
 \frac{|G(n,r,\preceq_q)|}{\binom nr}
 \le
 \frac{2}{k+1}\left(1+\frac1q\right)^{k+1}+o(1)
 =\frac{2}{k+1}+o(1).
$$
\end{lemma}

\begin{proof}
By Lemma~\ref{lem:turan-system}, the $r$-graph $G(n,r,\preceq_q)$ is a Tur\'an $(n,r+1,r)$-system. It remains to estimate its size. Choose $e=\{x_1<\cdots<x_r\}$ uniformly at random from $\binom{[n]}r$. Let $E_L$ be the event that
$
 f(L_{k-1}(e))
 \le
 f(R_{k}(e)\setminus\{x_i\})
$
for every $x_i\in R_{k}(e)$,
and let $E_R$ be the event that
$
 f(R_{k-1}(e))
 \le
 f(L_{k}(e)\setminus\{x_i\})
$
for every $x_i\in L_{k}(e)$.
Since $\preceq_q$ orders sets first by their $f$-values, every member of $G(n,r,\preceq_q)$ satisfies $E_L$ or $E_R$. Therefore
$$
 \frac{|G(n,r,\preceq_q)|}{\binom nr}
 \le\bbP(E_L)+\bbP(E_R).
$$

We first estimate the probability $\bbP(E_L)$. Let
$
 u:=f(L_{k-1}(e))
$
and, after writing $R_{k}(e)=\{z_1<\cdots<z_k\}$, let
$v_i:=f(R_{k}(e)\setminus\{z_i\})$ for $i\in[k]$.
The sets $L_{k-1}(e)$ and $R_{k}(e)$ are disjoint. If the residues of $x_1,\ldots,x_r$ were independent and uniform modulo $q$, then $u$ would be uniform on $[q]$. Moreover, by Lemma~\ref{lem:deletion-map}, the vector $(v_1,\ldots,v_k)$ would be uniform on $[q]^k$ and independent of $u$. By Fact~\ref{fact:residues}, the actual joint distribution differs from this uniform distribution by a multiplicative factor $1+O_r(q/n)$, uniformly over all residue patterns. More explicitly, the map from the full residue vector of $e$ to $(u,v_1,\ldots,v_k)$ has equally sized fibres: the sum defining $u$ is uniform, the deletion map is bijective, and any unused residue is free. Consequently,
\[
 \begin{aligned}
 \bbP(E_L)
 =\frac{1+O_r(q/n)}{q^{k+1}}
   \sum_{j=1}^q j^k
 \le
 \frac1{k+1}\left(1+\frac1q\right)^{k+1}+o(1),
 \end{aligned}
\]
where we used $\sum_{j=1}^qj^k\le\int_0^{q+1}x^k\,dx$ and $q=o(n)$.

By symmetry, the same bound applies to $\bbP(E_R)$.
Combining these two inequalities proves the first assertion.
Since $q\to\infty$, the second assertion follows.
\end{proof}
Now we prove \cref{thm:main-explicit}.
\begin{proof}[Proof of \cref{thm:main-explicit}]
For $r\ge3$, the result follows from Lemma~\ref{lem:explicit}, since
$k=\lfloor(r+1)/2\rfloor$. If $r=2$, then we can take all members of
$\binom{[n]}2$; this is a direct, explicit Tur\'an $(n,3,2)$-system
and gives the stated bound. 

It remains to verify the assertion about the computational complexity of deciding the inclusion of a given $r$-set in the Tur\'an system. For $r\ge3$, the integer $q$ and the order $\preceq_q$ are determined by $n$ and $r$. Given an $r$-set $e\subseteq[n]$, first sort its elements as $e=\{x_1<\cdots<x_r\}$. Membership in $G(n,r,\preceq_q)$ can then be decided by computing $O(r)$ subset sums modulo $q$ and making $O(r)$
lexicographic comparisons between sets of size $k-1$. This takes time polynomial in $r\log_2 n$. For $r=2$, membership is immediate.
This completes the proof.
\end{proof}

\section{The case when \texorpdfstring{$s\ge r+2$}{s >= r+2}}
\label{sed:LargerGap}
Finally, we discuss the case when $s=r+\rho$ with $\rho\ge 2$. For fixed $\rho\ge 2$, Frankl and R\"odl \cite{FranklRodl1985}
proved that
\[
 t(r+\rho,r)
 \le
 \frac{(1+o(1))\rho(\rho+4)\ln r}{\binom{r+\rho}{\rho}},\quad \text{as $r\to\infty$.}
\]
Pikhurko~\cite[Theorem~1.2 and Corollary~1.3]{Pikhurko2025} removed the factor $\ln r$, proving that
\[
 t(r+\rho,r)
 \le
 \frac{\mu_\rho+o(1)}{\binom{r+\rho}{\rho}},\quad \text{as $r\to\infty$,}
\]
where $\mu_\rho$ depends only on $\rho$ and can be chosen so that
$\mu_\rho=(1+o(1))\rho\ln \rho$ as $\rho\to\infty$. Liu and Pikhurko 
\cite[Theorem~2(ii)]{LiuPikhurko2026} further
extended this estimate to every function $\rho=\rho(r)$ satisfying
$\rho\to\infty$ and $\rho=o(\sqrt r)$. Bounds for larger $\rho$ (containing again $\ln r$ as a factor) can be found in \cite{LiuPikhurko2026, Sidorenko1997}.

The argument in \cref{sec:proof} can be extended to every $\rho$, giving the following.

\begin{proposition}\label{prop:general-R}
Let $n,r,\rho$ be positive integers such that $n\ge r+\rho$ and
$r\ge \rho\ge2$. Then
\[
 T(n,r+\rho,r)
 \le
 \frac{2}{1+\binom{\lfloor(r+\rho)/2\rfloor}{\rho}}\binom nr.
\]
\end{proposition}

\begin{proof}
Set $k=\lfloor(r+\rho)/2\rfloor$. If $k=\rho$, then the claimed inequality
is just the trivial bound $T(n,r+\rho,r)\le\binom nr$. Thus suppose that
$k>\rho$.

Recall that $L_\ell(X)$ and $R_\ell(X)$ denote respectively the initial and final $\ell$-subsets of $X\subseteq[n]$ with respect to the standard order $\le$ on $[n]$. Let $\preceq$ be an arbitrary total order on $\binom{[n]}{k-\rho}$.
Let $G_\rho(n,r,\preceq)$ denote the $r$-graph on $[n]$ where we include $e$ if at least one of the following conditions holds:
\begin{enumerate}
\renewcommand{\labelenumi}{(\roman{enumi})}
\item $L_{k-\rho}(e)\preceq Y$ for every
$Y\in\binom{R_{k}(e)}{k-\rho}$, or
\item $R_{k-\rho}(e)\preceq Z$ for every
$Z\in\binom{L_{k}(e)}{k-\rho}$.
\end{enumerate}

We first check that $G_\rho(n,r,\preceq)$ is a Tur\'an
$(n,r+\rho,r)$-system. Take any $(r+\rho)$-subset
$X=\{x_1<\cdots<x_{r+\rho}\}$ of $[n]$. Since $2k\le r+\rho$, the sets $L_k(X)$ and
$R_k(X)$ are disjoint. Let $W$ be the minimum member of
\[
 \binom{L_k(X)}{k-\rho}\cup\binom{R_k(X)}{k-\rho}
\]
with respect to $\preceq$. By symmetry, assume that
$W\in\binom{L_k(X)}{k-\rho}$. If $r+\rho=2k$, set $e:=W\cup R_k(X)$; if
$r+\rho=2k+1$, set $e:=W\cup\{x_{k+1}\}\cup R_k(X)$. Then $|e|=r$,
$L_{k-\rho}(e)=W$, and $R_{k}(e)=R_k(X)$, so the $\preceq$-minimality of $W$ implies that
$e\in G_\rho(n,r,\preceq)$. 

Finally, choose $\preceq$ uniformly at random. For a fixed
$e\in\binom{[n]}r$, the set $L_{k-\rho}(e)$ is disjoint from $R_{k}(e)$, so the
$1+\binom{k}{\rho}$ sets compared in Condition (i) are distinct.
Likewise, $R_{k-\rho}(e)$ is disjoint from $L_{k}(e)$, so the analogous statement
holds for Condition (ii). Hence each of the two defining
conditions has probability $1/{(1+\binom{k}{\rho})}$.
Consequently,
\[
 \bbE\,|G_\rho(n,r,\preceq)|
 \le
 \frac{2}{1+\binom{k}{\rho}}\binom nr.
\]
Take an ordering for which the number of edges in $G_\rho(n,r,\preceq)$ is at most its expected value.
\end{proof}

For every fixed $\rho\ge2$, Proposition~\ref{prop:general-R} gives
\[
 t(r+\rho,r)
 \le
 \frac{2^{\rho+1}+o_r(1)}{\binom{r+\rho}{\rho}},
 \quad\text{as }r\to\infty.
\]
For $\rho=2$, this gives the constant $8$, improving the constant
$9.267...$ coming from \cite[Theorem~1.2]{Pikhurko2025}. However, Proposition~\ref{prop:general-R} produces weaker bounds than those from \cite[Theorem~1.2]{Pikhurko2025} already from $\rho=3$.
Thus, for $\rho\ge 3$, the advantage of Proposition~\ref{prop:general-R} is in the simple and non-recursive nature of the construction.

\section*{Acknowledgements}

Jun Gao and Oleg Pikhurko were supported by ERC Advanced Grant 101020255; Peiru Kuang was supported by a subproject of the AI for Math and Science Program;
Yan Wang was supported by the National Key R\&D Program of China under Grant No.~2022YFA1006400 and by the National Natural Science Foundation of China under Grant No.~12571376. This work was initiated and partly carried out during the research visits funded by the SJTU--Warwick Joint Seed Fund.

\section*{AI disclosure}

ChatGPT 5.5 was able to autonomously improve (modulo some fixable gaps) the constant $4.911$ from~\cite{Pikhurko2025} to
$4.45...$ by suggesting a more complicated version of the construction from~\cite{Pikhurko2025}. Subsequent human--AI interaction led to the simple construction
presented here. AI tools were also used for improving the draft written by the authors. 

\bibliographystyle{abbrv}
\bibliography{bib}
\end{document}